\documentclass[11pt,reqno]{amsart}
\usepackage[comma]{natbib}
\usepackage[a4paper,left=1in,right=1in,top=1.25in,bottom=1.25in]{geometry}
\usepackage{color}
\usepackage{graphicx}
\usepackage{amsmath}
\usepackage{amssymb}
\usepackage{enumerate}
\usepackage{comment}

\usepackage[monochrome]{xcolor}

\newtheorem{theorem}{Theorem}
\newtheorem{corollary}[theorem]{Corollary}
\newtheorem{lemma}[theorem]{Lemma}
\theoremstyle{definition}
\newtheorem*{definition}{Definition}
\theoremstyle{remark}

\theoremstyle{remark}
\newtheorem*{example}{Example}

\begin{document}

\title[]{Lorenz ordering of parasite loads across hosts in Isham's model of Host-Macroparasite Interaction}
\author[]{R. McVinish}
\address{School of Mathematics and Physics, University of Queensland}
\email{r.mcvinish@uq.edu.au}

\subjclass[2020]{60E15; 92D25; 60E10}

\begin{abstract}

In Isham's model of host-macroparasite interaction, parasite-induced host mortality increases parasite aggregation in the sense of the Lorenz order and related measures when the distribution of the number of parasites entering the host at infectious contacts is log-concave. Furthermore, in the presence of parasite-induce host mortality, the rate of parasite mortality may no longer have a monotone effect on aggregation.

\end{abstract}

\keywords{aggregation; compound Poisson distribution; convex order; Gini index; Lorenz order}

\maketitle

\section{Introduction}

Nearly all parasite populations exhibit substantial variation in their distribution among hosts, a phenomenon known as aggregation \citep{Poulin:07}. While no universally accepted definition of aggregation exists, several measures are commonly used to quantify aggregation in parasite-host systems \citep{ML:2020, MPF:2023}. In mathematical models of parasite acquisition, aggregation is typically measured using either the variance-to-mean ratio (VMR) \citep{Isham:95, BP:00, HI:2000, Peacock:18} or the $k$ parameter of the negative binomial distribution \citep{AM:78a, AM:78b, RP:02, Schreiber:06, McPherson:12}. Both measures can be interpreted as quantifying departures from a Poisson distribution.

\citet{Poulin:1993} offered an alternative interpretation, arguing that parasitologists often view aggregation as deviation from a distribution in which all hosts carry the same parasite load. From this perspective, the Lorenz order \citep{Lorenz:1905, Arnold:87} and related measures become central. In particular, Poulin's {\em index of discrepancy} ($D$) \citep{Poulin:1993}, essentially the Gini index \citep{Gini:2005}, has become one of the standard measures in empirical studies. However, applications of the Lorenz order to  mathematical models of parasite acquisition are rare. One exception is the recent re-examination of Tallis-Leyton model \citep{TL:1969, McVinish:2025b}, a simple model in which parasites accumulate without affecting host mortality, or eliciting an immune response.

One of the aims in analysing mathematical models of parasite acquisition is to better understand the factors controlling the degree of aggregation. Parasite-induced host mortality (PIHM) is sometimes reported in the parasitology literature as reducing aggregation. The usual argument is that PIHM disproportionately removes heavily infected hosts, causing parasite loads to concentrate more tightly around the mean \citep{AG:82, Poulin:2011}. Supporting evidence for this claim appears limited. \citet{Isham:95} proposed a model for parasite acquisition with PIHM in which the host's mortality rate increases linearly with parasite load. In the special case where parasites enter the host one at a time following a Poisson process, the parasite load conditioned on host survival has a Poisson distribution and the resulting VMR is one regardless of PIHM strength. \citet{BP:00} consider related models with two different PIHM rate functions and showed that the VMR may exceed or fall below one depending on the form of the PIHM. 

Building on our recent analysis of the Tallis–Leyton model \citep{McVinish:2025b}, we apply the Lorenz-based perspective on aggregation to Isham’s model. Our results show that, under mild assumptions, PIHM increases aggregation and its presence can alter the effects of two other key parameters: the parasite mortality rate and the distribution of parasites entering the host at infectious contact.

\section{Preliminaries}

\subsection{Notation}
Let's first establish some basic notation. The probability mass function, distribution function, and survival function of a distribution will be denoted by $\pi$, $F$, and $\overline{F}$, respectively. The probability generating function of a distribution with probability mass function $\pi$ is $ G(z) = \sum_{i=0}^\infty \pi(i) z^i$. Subscripts will be used to denote the random variable to which these functions are associated.

\subsection{Isham's model} \label{sec:Isham}

\cite{Isham:95} proposed the following model for the parasite load of a host that has survived to age $a$. At birth the host is parasite free. During its lifetime, the host has infectious contacts at times following a homogeneous Poisson process with rate $\phi$. At the infectious contacts the host acquires a random number of parasites $C$. The number of parasites acquired at infectious contacts form a sequence of independent random variables. Acquired parasites die independently at rate $\mu_M$ per parasite. Let $\mu_H(a)$ be the mortality rate of the host age $a$ in the absence of any parasites. The host mortality rate is increased by $\alpha$ for each parasite acquired. For a host that has survived to age $a$ with a parasite load $m$, the transitions are
\begin{align*}
    m  \to m+k & \quad  \text{at rate } \phi \, \pi_C(k) \\
    m  \to m-1 & \quad  \text{at rate } \mu_M \\
    \text{host dies} & \quad  \text{at rate } \mu_H(a) + \alpha m 
\end{align*}
It is assumed throughout that $\mu_M + \alpha > 0$. Let $M(a)$ denote the parasite load of a host conditional upon the host's survival to age $a$. Isham showed that the PGF of $M(a)$ is 
\[
G_M(a;z) = \exp\left(\int_{\theta(a;1)}^1  \frac{\phi (1 - G_C(u))}{(\mu_M + \alpha)u - \mu_M} \, du + \int_{\theta(a;z)}^z \frac{\phi (1 - G_C(u))}{(\mu_M + \alpha)u - \mu_M} \, du \right),
\]
where 
\[
\theta(a;z) =  \lambda + (z-\lambda) \frac{ \exp( -(\alpha + \mu_M)a)}{(\alpha + \mu_M)}, \quad \text{and} \quad \lambda = \frac{\mu_M}{(\alpha + \mu_M)}.
\]
Of particular interest in this paper is the limiting case as $a \to \infty$ which has PGF 
\begin{equation}
G_M(\infty;z) = \exp\left(\frac{\phi}{\alpha + \mu_M} \int_z^1 \frac{G_C(\lambda) - G_C(u)}{u - \lambda} du \right).  \label{eq:pgf_limit}
\end{equation}
Convergence to the equilibrium is generally fast, due to the exponential decay in $\theta(a;z)$. 
Determining the distribution from (\ref{eq:pgf_limit}) is generally difficult, though Isham provided explicit expressions in two cases. 

\begin{example}
Suppose the parasites enter the host one at a time so $C=1$, almost surely. In this case $M(a)$ has a Poisson distribution with mean 
\begin{equation}
    \frac{\phi \left(1 - \exp \left( -(\mu_M + \alpha) a\right) \right)}{(\mu_M + \alpha)}.  \label{eq:poisson-mean}  
\end{equation}  
\end{example}

\begin{example}
Suppose the number of parasites entering the host at an infectious contact follows a Geometric distribution. In this case only the limiting distribution $M(\infty)$ has an explicit expression. Let $\mathsf{NB}(m,\nu)$ denote the negative binomial distribution with mean $m$ and variance $ m + m^2/\nu$. Isham showed that if $C \sim \mathsf{NB}(m_C, 1)$, then the distribution of $M(\infty)$ is 
\begin{equation}
\mathsf{NB}\left(\frac{\phi m_C}{\mu_M + \alpha + \alpha m_C}, \frac{\phi}{\mu_M + \alpha + \alpha m_C} \right). 
\label{eq:negative-binomial}  
\end{equation}
\end{example}

VMR is a standard measure of aggregation used in mathematical models in parasitology. The above examples show that for any values of $\alpha, \mu_M, \phi$, $\text{VMR}=1$ when $C=1$, almost surely, and $\text{VRM}=1+m_C$ when $C\sim \mathsf(m_C,1)$. In general, the variance and mean of $M(\infty)$ can be obtained from the PGF by differentiation, leading to the following expressions for VMR:
\begin{equation}
    \text{VMR} = \frac{G_C^\prime(1)}{1- G_C(\lambda)} - \frac{\mu_M}{\alpha} \label{eq:VMR}
\end{equation}
for $\lambda <1$, and 
\begin{equation}
    \text{VMR} = 1 + \frac{G^{\prime\prime}_C(1)}{2G^\prime_C(1)} \label{eq:VMR2}
\end{equation}
for $ \lambda = 1$. 

\subsection{Convex order and Lorenz order} \label{sec:cx}

The convex order is a partial order on the space of distributions that compares the variability of two distribution with the same expectation.

\begin{definition}
    Let $X$ and $Y$ be two random variables such that $\mathbb{E} \left[X\right] = \mathbb{E} \left[Y\right]$. We say $X$ is smaller than $Y$ in the {\em convex order}, denoted $X \leq_{\rm cx} Y$, if $ \mathbb{E} \left[\psi(X)\right] \leq \mathbb{E} \left[\psi(Y)\right]$ for all convex functions $ \psi: \mathbb{R} \to \mathbb{R}$, provided the expectations exist. 
\end{definition}

\citet[Section 3.A]{SS:07} provide an extensive review of results on the convex order. The distributions of interest in this analysis don't necessarily have the same mean so we will apply a related stochastic ordering that can be used to compare distributions with different means, provided they are supported on the non-negative reals with finite mean, called the Lorenz order \citep[Chapter 3]{Arnold:87}. The Lorenz order is based on the Lorenz curve, a graphical measure of wealth inequality proposed by \citet{Lorenz:1905}. Modern presentations of the Lorenz curve adopt the definition given by \citet{Gastwirth:1971}.

\begin{definition}
    The {\em Lorenz curve} $L: [0,1] \to [0,1] $ for the distribution $F$ on $[0,\infty)$ with finite mean $ \mu$ is given by
    \begin{equation}
        L(u) = \frac{\int_{0}^{u} F^{-1}(y)\, dy}{\mu}, 
    \end{equation}
    where $ F^{-1}$ is the quantile function
    \begin{equation}
        F^{-1}(y) = \sup \{ x : F(x) \leq y \} \quad \text{for } y \in (0,1).
    \end{equation}
\end{definition}

When a distribution places all probability on a single point, the corresponding Lorenz curve is given by $L(u) = u$. This is called the egalitarian line. A Lorenz curve never rises above the egalitarian line, that is $L(u) \leq u$ for all $ u \in [0,1]$. We can now state the definition of Lorenz order \citep[Definition 3.2.1]{Arnold:87}.

\begin{definition}
    Let $ X$ and $Y$ be random variables with the respective Lorenz curves denoted $ L_{X} $ and $ L_{Y}$. We say $ X$ is  smaller in the {\em Lorenz order}, denoted $ X \leq_{\rm{Lorenz}} Y $ if $ L_{X}(u) \geq L_{Y}(u)$ for every $ u \in [0,1]$.
\end{definition} 

Unless stated otherwise, any reference to increasing or decreasing aggregation are understood in terms of the Lorenz order of the distributions.

In our analysis of Isham's model we will not work directly with the Lorenz curve, instead we will exploit the connection between the Lorenz order and convex order:

\begin{align}
X &\leq_{\rm Lorenz} Y    & \text{is equivalent to}&  &\frac{X}{\mathbb{E} \left[X\right]} &\leq_{\rm cx} \frac{Y}{\mathbb{E} \left[Y\right]}. \label{eq:cx_lorenz_equiv}
\end{align}
As in \citet{McVinish:2025b}, compound Poisson distributions will play a large role in establishing the Lorenz order between two parasite-host systems. The following result will be used on several occasions.

\begin{lemma} \label{lem:CP}
    Let $N$ and $\tilde{N}$ be two Poisson random variables such that $\mathbb{E} \tilde{N} < \mathbb{E} N$, and let $X_1,X_2,\ldots$ and $\tilde{X}_1,\tilde{X}_2,\ldots$ be two independent and identically distributed sequences of random variable such that $X_i \leq_{\rm cx} \tilde{X}_i$ for all $i$. Then
    \[
    \sum_{i=1}^N X_i \leq_{\rm Lorenz} \sum_{i=1}^{\tilde{N}} \tilde{X}_i.
    \]
\end{lemma}

\begin{proof}
    Let $p = \mathbb{E} \tilde{N}/\mathbb{E} N $ and define $B_1,B_2,\ldots$ be a sequence of independent $\mathsf{Bernoulli}(p)$ random variables. By \citet[Theorem 3.A.33]{SS:07}, $p X_i \leq_{\rm cx} B_i X_i$. Applying \citet[Theorem 3.A.12(d)]{SS:07} yields  
    \[
    p\sum_{i=1}^n X_i \leq_{\rm cx} \sum_{i=1}^n B_i X_i
    \]
    for any $n \in \mathbb{N}_0$. Applying \citet[Theorem 3.A.12(b)]{SS:07} and using the standard binomial thinning property of the Poisson distribution yields  
    \begin{equation}
    p\sum_{i=1}^N X_i \leq_{\rm cx} \sum_{i=1}^N B_i X_i \stackrel{d}{=} \sum_{i=1}^{\tilde{N}} X_i.  \label{eq:cp1}
    \end{equation}
    If $X_i \leq_{\rm cx} \tilde{X}_i$, then applying \citet[Theorem 3.A.12(b,d)]{SS:07} again we see 
    \begin{equation}
    \sum_{i=1}^{\tilde{N}} X_i \leq_{\rm cx} \sum_{i=1}^{\tilde{N}} \tilde{X}_i. \label{eq:cp2}
    \end{equation}
    The result now follows from (\ref{eq:cp1}), (\ref{eq:cp2}), and the connection between convex and Lorenz orders (\ref{eq:cx_lorenz_equiv}).
\end{proof}

If $X \leq_{\rm Lorenz} Y$ , then a measure of aggregation $I(\cdot)$ respecting the Lorenz order  satisfies $I(X) \leq I(Y)$. Several such measures are reviewed in \citet[Chapter 5]{Arnold:87}. Usually, the simplest measure to calculate that respects the Lorenz order is the coefficient of variation (CV). The following expressions give the CV of $M(\infty)$ for Isham's model:
\begin{equation}
\text{CV} = \sqrt{\frac{\alpha \, G_C^\prime(1)}{\phi \, (1-G_C(\lambda))^2} - \frac{\mu_M}{\phi \, (1-G_C(\lambda))}} \label{eq:cv2a}
\end{equation}
for $\lambda < 1$, and 
\begin{equation}
\text{CV} = \frac{1}{G_C^\prime(1)}\sqrt{\frac{\mu_M}{\phi} \left(G_C^\prime(1) + \tfrac{1}{2} G^{\prime\prime}_C(1)\right)}    \label{eq:cv2b}
\end{equation}
for $\lambda = 1$. Though not usually considered as such, $1 - \text{prevalence}$, that is the probability that the host is parasite free, is a measure of aggregation that respects the Lorenz order. For Isham's model, this is given by
\begin{equation}
    \mathbb{P}(M(\infty) = 0) = \exp\left(\frac{\phi}{\alpha + \mu_M} \int_0^1 \frac{G_C(\lambda) - G_C(u)}{u - \lambda} du \right). \label{eq:p0}
\end{equation}
Calculation of other measures of aggregation respecting the Lorenz order like the Gini index\citep{Gini:2005}, given by 
\[
\frac{\mathbb{E}\left|M(\infty) - M'(\infty)\right|}{2\mathbb{E}M(\infty)}, 
\]
where $M'(\infty)$ is an independent copy of $M(\infty)$, and the Pietra index \citep{Pietra:2014}, given by
\[
\frac{\mathbb{E}\left|M(\infty) - \mathbb{E}M(\infty)\right|}{2\mathbb{E}M(\infty)}, 
\]
require the numerical inversion of the PGF, which may be achieved using the algorithm described in \citet{AB:92}.

\section{Main Results}
\subsection{Compound Poisson representation}

Our first result shows that the distribution with PGF (\ref{eq:pgf_limit}) can be represented as a compound Poisson distribution where the distribution of the summands is a mixture and each component distribution has expectation equal to one. 

\begin{theorem} \label{thm:rep}
For each $j \in \mathbb{N}_0$, define $\beta_j = \sum_{i=j+1}^\infty \pi_C(i)$ and the function $\omega_j :[0,1] \to [0,1]$ such that 
\begin{equation}
    \omega_j(\lambda)  = \left\{ \begin{array}{ll}
    \frac{\beta_j (1-\lambda)\lambda^j}{1-G_C(\lambda)}, & \lambda \in [0,1) \\
    \frac{\beta_j}{G_C^{\prime}(1)}, & \lambda = 1
    \end{array} \right. 
    . \label{eq:thm_omega}
\end{equation}
If $\beta_j \neq 0$, define the function $G_j : [0,1] \to [0,1]$ such that
    \begin{equation}
        G_{j}(z)  = 1 -  \int_z^1 \frac{G_C(u) - \sum_{i=0}^{j} \pi_C(i) u^i}{\beta_j u^{j+1}} du,  \label{eq:thm_G}
    \end{equation}
Then (i) each $G_j(z)$ is a valid PGF with $G^\prime_j(1) = 1$, (ii) for any $\lambda \in [0,1]$, $\sum_{j=0}^\infty \omega_j(\lambda) = 1$, and (iii) for $\lambda \in [0,1)$,
\begin{equation}
G_M(\infty;z)  = \exp\left(\frac{\phi(1-G_C(\lambda))}{\alpha} \sum_{j=0}^\infty \omega_j(\lambda) \left(G_j(z) -1\right) \right)
\label{eq:pgf_rep}
\end{equation}
and for $\lambda = 1$, 
\begin{equation}
G_M(\infty;z)  = \exp\left(\frac{\phi\,  G^\prime_C(1)}{\mu_M} \sum_{j=0}^\infty \omega_j(1) \left(G_j(z) -1\right) \right).
\label{eq:pgf_repb}
\end{equation}

\end{theorem}

\begin{proof}[Proof of Theorem \ref{thm:rep}]
We prove the claims (i)-(iii) of Theorem \ref{thm:rep} separately.

(i) We first determine the power series expansion of $G_j(z)$:
    \begin{align}
        G_j(z) &= 1 - \int_z^1 \frac{G_C(u) - \sum_{i=0}^{j} \pi_C(i) u^i}{\beta_j u^{j+1}} du \nonumber \\
        & = 1 - \int_z^1 \beta_j^{-1} \sum_{i=0}^\infty \pi_C(i+j+1) \, u^{i}\, du \nonumber\\
        & = 1 -  \sum_{i=0}^\infty \frac{\pi_C(i+j+1)(1 - z^{i+1} )}{\beta_j\,(i+1)} \nonumber \\  
        & = 1 -   \sum_{i=0}^\infty \frac{\pi_C(i+j+1)}{\beta_j\, (i+1)} +  \sum_{i=0}^\infty \frac{\pi_C(i+j+1)}{\beta_j\, (i+1)} z^{i+1}. \label{eq:pmf_pgf}
    \end{align}
    As  $\beta_j^{-1} \sum_{i=0}^\infty \pi_C(i+j+1) =1$, it follows that $1 - \sum_{i=0}^\infty \frac{\pi_C(i+j+1)}{\beta_j\, (i+1)} \geq 0$. This power series has non-negative coefficients and the coefficients sum to one, so $G_j(z)$ is a valid PGF. Furthermore,
    \[
    G^{\prime}_j(1) = \frac{G_C(1) - \sum_{i=0}^{j} \pi_C(i) }{\beta_j} = \beta_j^{-1} \sum_{i=j+1}^{\infty} \pi_C(i) = 1.
    \]

(ii) The tail-sum formula for expectations and the fact $G_C^{\prime}(1) = \mathbb{E}\, C$ show that $\sum_{j=0}^\infty \omega_j(\lambda) = 1$ when $\lambda = 1$. In the case where $\lambda < 1$, 
    \[
    \sum_{j=0}^\infty \omega_j (\lambda) = \frac{1-\lambda}{1-G_C(\lambda)}\sum_{j=0}^\infty   \sum_{i=j+1}^\infty \pi_C(i) \lambda^j.
    \]
    Interchanging the order of summation and applying the geometric summation formula yields
    \[
    \sum_{j=0}^\infty \omega_j (\lambda) = \frac{1-\lambda}{1-G_C(\lambda)}\sum_{j=1}^\infty    \pi_C(j) \frac{1 - \lambda^j}{1-\lambda} = \frac{1}{1-G_C(\lambda)}\sum_{j=1}^\infty    \pi_C(j) (1 - \lambda^j).
    \]
    If $ \lambda = 0 $, then the sum is equal to one since $1-G_C(0) = \sum_{j=1}^\infty \pi_C(j)$. If $\lambda \in (0,1)$, then as $\pi_C(0)(1-\lambda^0) = 0$,
    \[
    \sum_{j=0}^\infty \omega_j (\lambda)= \frac{1}{1-G_C(\lambda)}\sum_{j=0}^\infty    \pi_C(j) (1 - \lambda^j) = 1.
    \]

    (iii) In the case where $\lambda < 1$,
    \begin{align*}
        \frac{\phi\, (1-G_C(\lambda))}{\alpha}\sum_{j=0}^\infty \omega_j(\lambda) \left(G_j(z)  -1 \right) 
        & =  - \frac{\phi\, (1-\lambda)}{\alpha} \sum_{j=0}^\infty \lambda^j \int_z^1 \frac{G_C(u) - \sum_{i=0}^{j} \pi_C(i) u^i}{u^{j+1}} du \\
        & =  - \frac{\phi}{\alpha + \mu_M} \sum_{j=0}^\infty \lambda^j \int_z^1 \sum_{i=j+1}^\infty \pi_C(i) u^{i-j-1} du \\
        & =  - \frac{\phi}{\alpha + \mu_M}  \int_z^1 \sum_{j=0}^\infty \sum_{i=j+1}^\infty \lambda^j \,   \pi_C(i) u^{i-j-1} du,
    \end{align*}
    where we have interchanged the order of integration and summation. 
    \begin{align*}
        \frac{\phi(1-G_C(\lambda))}{\alpha}\sum_{j=0}^\infty \omega_j(\lambda) \left(G_j(z) - 1 \right) & =  - \frac{\phi}{\alpha + \mu_M} \int_z^1 \sum_{i=1}^{\infty} \pi_C(i) u^{i-1} \sum_{j=0}^{i-1} (\lambda/u)^j  du \\
        & =  - \frac{\phi}{\alpha + \mu_M} \int_z^1 \sum_{i=1}^{\infty} \pi_C(i) u^{i-1} \frac{1 - (\lambda/u)^i}{1- (\lambda/u)}   du \\
        & = - \frac{\phi}{\alpha + \mu_M} \int_z^1 \sum_{i=1}^{\infty} \pi_C(i)  \frac{u^i - \lambda^i}{u- \lambda}   du \\
        & = \frac{\phi}{\alpha + \mu_M}\int_z^1  \frac{ G_C(\lambda) - G_C(u)}{u- \lambda}   du.
    \end{align*}
    This proves the equation (\ref{eq:pgf_rep}) for $\lambda <1$. In the case where $\lambda = 1$, take a sequence $\alpha_n \to 0$ and $\mu_{M,n} \to \mu_M$ such that $\mu_{M,n}/(\alpha_n + \mu_{M,n}) = \lambda_n \to 1$. The result then follows since
    \begin{align*}
        \lim_{n \to \infty} \frac{1 - G_C(\lambda_n)}{\alpha_n} = \lim_{n \to \infty} \frac{1-G_C(\lambda_n)}{1-\lambda_n} \frac{1-\lambda_n}{\alpha_n} = \frac{G^\prime_C(1)}{\mu_M}
    \end{align*}
    and for each $j \in \mathbb{N}_0$
    \begin{align*}
        \lim_{n\to\infty} \omega_j(\lambda_n) = \lim_{n\to\infty} \frac{1-\lambda_n}{1-G_C(\lambda_n)} \lambda_n^j \beta_j = \frac{\beta_j}{G_C^\prime(1)}.
    \end{align*}
\end{proof}

The distributions with PGF (\ref{eq:thm_G}) are fully determined by the distribution of $C$. The PIHM rate $\alpha$ and parasite mortality rate $\mu_M$ influence only the distribution of the summands in the compound Poisson representation via the weights $\omega_j(\lambda)$. In the following examples explicit expressions for the component distributions are derived for some specific distributions of $C$. In these cases the component distributions $F_j$ can be compared in the convex order.

\begin{example}
    Suppose the distribution of $C$ is supported on $\{0,1,2\}$ with $\pi_C(2) > 0$. Direct calculation shows
    \begin{align*}
        G_0(z) 
        & = \frac{\pi_C(2)}{2(\pi_C(1) + \pi_C(2))} + \frac{\pi_C(1)}{\pi_C(1) + \pi_C(2)} z + \frac{\pi_C(2)}{2(\pi_C(1) + \pi_C(2))} z^2
    \end{align*}
    and
    \begin{align*}
        G_1(z) 
        & = z.
    \end{align*}
    Therefore, $F_0$ is a symmetric distribution supported on $\{0,1,2\}$ while $F_1$ is the point mass at 1. From the definition of the convex order and Jensen's inequality, we see that $F_1 \leq_{\rm cx} F_0$. 
\end{example}

\begin{example}
    Suppose $C$ has a $\mathsf{Geometric}(p)$ distribution. Then
    \[
    \beta_j = \sum_{i=j+1}^\infty p (1-p)^{i} = (1-p)^{j+1}
    \]
    and
    \begin{align*}
        G_{j}(z) 
        & = 1 -  \int_z^1 \left(\beta_j \, u^{j+1}\right)^{-1}\sum_{i=j+1}^\infty p(1-p)^{i} u^{i} du\\
        & = 1 - \int_z^1 \frac{p}{1 - (1-p)u} du  \\
        & = 1 - \frac{p}{1-p} \ln\left(\frac{1 - (1-p)z}{p} \right).
    \end{align*}
    This is the PGF of a `zero-inflated' logarithmic series distribution. Importantly, this distribution does not depend on $j$, so $F_{j+1} \stackrel{d}{=} F_j$ for each $ j \in \mathbb{N}_0$.
\end{example}

Explicit computation of the component distributions appears to be feasible in only a limited number of cases. Nevertheless, under the additional assumption that the distribution of $C$ is either log-concave or log-convex on $\mathbb{N}_0$, we can still establish the convex ordering between component distributions $F_j$. Recall a distribution on $\mathbb{N}_0$ is said to be {\em log-concave} if the probability mass function $\pi$ satisfies $\pi(k)^2 \geq \pi(k-1) \pi(k+1)$ for all $k \geq 1$. If the probability mass function satisfies $\pi(k)^2 \leq \pi(k-1) \pi(k+1)$ for all $k \geq 1$, the distribution is said to be {\em log-convex}.

\begin{theorem} \label{thm:Fcx}
    If the distribution of $C$ is log-concave, then $F_{j+1} \leq_{\rm cx} F_{j}$ for any $ j \in \mathbb{N}_{0}$ with $\beta_{j+1} > 0$. If the distribution of $C$ is log-convex, then $F_{j} \leq_{\rm cx} F_{j+1}$ for any $ j \in \mathbb{N}_{0}$ with $\beta_{j+1} >0$.         
\end{theorem}

\begin{proof}
    The proof is given for the case where the distribution of $C$ is log-concave. The case where the distribution of $C$ is log-convex is proved similarly. 
    
    The survival function $\overline{F}_j$ can be determined from (\ref{eq:pmf_pgf}) as
    \[
    \overline{F}_j(k) =  \sum_{i=k+1}^\infty \frac{\pi_C(i+j)}{i \beta_j}, \quad k \geq 0.
    \]
    If the sequence $\{\overline{F}_j(k) - \overline{F}_{j+1}(k): k \geq 0\}$ has a single sign change from negative to positive, then $F_{j+1} \leq_{\rm cx} F_j$  \citep[Theorem 3.A.44]{SS:07}. For $k \geq 0$,
    \begin{align*}
        \overline{F}_j(k) - \overline{F}_{j+1}(k) & = \sum_{i=k+1}^{\infty} i^{-1} \left(\frac{\pi_C(i+j)}{\beta_j} - \frac{\pi_C(i+j+1)}{\beta_{j+1}} \right).
    \end{align*}
    As the distribution of $C$ is log-concave, the sequence $\{\pi_C(i+1)/\pi_C(i): i \geq 0\}$ is non-increasing. Therefore, the sequence
    \[
    \left\{\frac{\pi_C(i+j)}{\beta_j} - \frac{\pi_C(i+j+1)}{\beta_{j+1}}: i \geq 1 \right\}
    \]
    can have at most one sign change and this must be from negative to positive since 
    \[
    \text{sign}\left(\frac{\pi_C(i+j)}{\beta_j} - \frac{\pi_C(i+j+1)}{\beta_{j+1}} \right) = \text{sign} \left(\frac{\beta_{j+1}}{\beta_{j}} - \frac{\pi_C(i+j+1)}{\pi_C(i+j)} \right). 
    \]
    This implies the sequence $\{\overline{F}_j(k) - \overline{F}_{j+1}(k): k \geq 0\}$ has at most a single sign change and this must be from negative to positive. If the sequence $\{\overline{F}_j(k) - \overline{F}_{j+1}(k): k \geq 0\}$ has no sign change, then either $F_{j+1} \stackrel{d}{=} F_j$ or the distributions $F_j$ and $F_{j+1}$ must have different expectations by the tail sum formula for expectations. However, each distribution $F_j$ has expectation one. Hence, unless $F_{j+1} \stackrel{d}{=} F_j$, there is one sign change in the sequence $\{\overline{F}_j(k) - \overline{F}_{j+1}(k): k \geq 0\}$ from negative to positive. Therefore, $F_{j+1} \leq_{\rm cx} F_j$.
\end{proof}


In the following subsections we apply these results to understand how the parameters of this model affect the aggregation of parasites in the host. When comparing two parasite-host systems, we will use a tilde to denote parameters from the second system.

\subsection{Parasite-induce host mortality} The examples in Section \ref{sec:Isham} show that there are cases where the PIHM parameter $\alpha$ has no effect on VMR. These same examples show that increasing $\alpha$ can {\em increase} aggregation in the sense of the Lorenz order since
\[
\mathsf{Poisson}(m_1) \leq_{\rm Lorenz} \mathsf{Poisson}(m_2) \quad \text{for } m_2 \leq m_1
\]
and 
\[
\mathsf{NB}(m_1, k_1) \leq_{\rm Lorenz} \mathsf{NB}(m_2,k_2) \quad \text{for } m_2 \leq m_1 \text{ and } k_2 \leq k_1.
\]
See, for example, \citet{ML:2024}. Furthermore, in these two examples CV is a strictly increasing function of $\alpha$ since VMR is constant but the expectation is a strictly decreasing function of $\alpha$. In the absence of parasite mortality, that is when $\mu_M =0$, increasing PIHM always increases aggregation. 

\begin{corollary} \label{cor:mu0}
    If $\alpha < \tilde{\alpha}$ and all other model parameters are equal with $\mu_M = 0$, then $M(\infty) \leq_{\rm Lorenz}  \tilde{M}(\infty) $.
\end{corollary} 
    
\begin{proof}
    If $\mu_M = 0$, then $\lambda = \tilde{\lambda} = 0$ for any values of $\alpha$ and $\tilde{\alpha}$. Let 
    \[
    N \sim \mathsf{Poisson}\left(\frac{\phi (1-G_C(0))}{\alpha}\right) \quad \text{and} \quad \tilde{N} \sim \mathsf{Poisson} \left(\frac{\phi (1-G_C(0)}{\tilde{\alpha}}\right).
    \] 
    By Theorem \ref{thm:rep}, $M(\infty) \stackrel{d}{=} \sum_{i=1}^N X_i$ and $\tilde{M}(\infty) \stackrel{d}{=} \sum_{i=1}^{\tilde{N}} X_i$, where $X_1, X_2, \ldots$ is a sequence of independent random variables having distribution $F_0$. The result now follows by application of Lemma \ref{lem:CP}.
\end{proof}

When parasite mortality is present, $0<\alpha < \tilde{\alpha}$ implies $0< \tilde{\lambda} <  \lambda < 1$. The distribution of the summands in the compound Poisson representation is now different for the two parasite-host systems. With the additional assumption that the distribution of $C$ is log-concave, we can still show that increasing $\alpha$ increases aggregation.
 
\begin{theorem} \label{thm:alpha}
     Assume the distribution of $C$ is log-concave. If $\alpha < \tilde{\alpha}$ and all other model parameters are equal, then $M(\infty) \leq_{\rm Lorenz} \tilde{M}(\infty)$.
\end{theorem}

\begin{proof} 
    Let 
    \[
    N \sim \mathsf{Poisson}\left(\frac{\phi (1-G_C(\lambda))}{\alpha}\right) \quad \text{and} \quad \tilde{N} \sim \mathsf{Poisson} \left(\frac{\phi (1-G_C(\tilde{\lambda}))}{\tilde{\alpha}}\right).
    \]
    Let $X_1,X_2,\ldots$ and $\tilde{X}_1, \tilde{X}_2,\ldots $ be two sequences of independent random variables such that 
    \[
    X_i \sim \sum_{j=0}^\infty \omega_j(\lambda) F_j, \quad \text{and} \quad \tilde{X}_i\sim \sum_{j=0}^\infty \omega_j(\tilde{\lambda}) F_j.
    \]
    By Theorem \ref{thm:rep} $M(\infty) \stackrel{d}{=} \sum_{i=1}^N X_i$ and $\tilde{M}(\infty) \stackrel{d}{=} \sum_{i=1}^{\tilde{N}} \tilde{X}_i$. The result follows from Lemma \ref{lem:CP} if we can show (i) $\mathbb{E} \tilde{N} < \mathbb{E} N$ and (ii) $X_i \leq_{\rm cx} \tilde{X}_i$.
    
    (i) Recall $\lambda = \mu_M/(\alpha+ \mu_M)$. We want to show $(1 - G_C(\lambda))/\alpha$ is a decreasing function of $\alpha$. Differentiating with respect to $\alpha$, 
    \begin{align}
        \frac{d}{d\alpha} \frac{(1- G_C(\lambda))}{\alpha} & = \frac{-G_C^{\prime}(\lambda) \frac{d\lambda}{d\alpha} \alpha - (1-G_C(\lambda))}{\alpha^2} \nonumber\\
        & = \frac{G_C^{\prime}(\lambda) \lambda (1-\lambda) - (1-G_C(\lambda))}{\alpha^2}. \label{eq:rate_deriv}
    \end{align}
    Taking the Taylor expansion of $G_C(1)$ about $\lambda$,
    \begin{align*}
        1 - G_C(\lambda) = G_C(1) - G_C(\lambda) & =  G_C^{\prime}(\lambda)(1-\lambda) + \sum_{k=2}^{\infty} G_C^{(k)}(\lambda) \frac{(1-\lambda)^{k}}{k!}.
    \end{align*}
    Substituting this expression into (\ref{eq:rate_deriv}) yields
    \[
    \frac{d}{d\alpha} \frac{(1- G_C(\lambda))}{\alpha}  = - \frac{1}{\alpha^2} \left(G_C^{\prime}(\lambda) (1-\lambda)^2 + \sum_{k=2}^{\infty} G_C^{(k)}(\lambda) \frac{(1-\lambda)^{k}}{k!}  \right)<0,
    \]
    since $G_C(\cdot)$ is an absolutely monotone function. This shows $ \mathbb{E} \tilde{N} < \mathbb{E} N$.  

    (ii) We want to show that for any convex function $\psi$, $\mathbb{E} \psi(X_i) \leq \mathbb{E} \psi(\tilde{X}_i)$, provided the expectations exist. 
    By Theorem \ref{thm:Fcx}, $\int \psi(x) F_j(dx)$ is a decreasing function of $j$ for any convex function $\psi$. For each $\lambda \in [0,1]$,  $\{\omega_j(\lambda), j \in \mathbb{N}_0\}$ forms a probability mass function on $\mathbb{N}_0$ by Theorem \ref{thm:rep}.  As $\alpha <\tilde{\alpha}$, $\tilde{\lambda} < \lambda$ and
    \[
    \frac{\omega_j(\lambda)}{\omega_j(\tilde{\lambda})} = \frac{(1-\lambda)(1-G_C(\tilde{\lambda}))}{(1-\tilde{\lambda})(1-G_C(\lambda))} \left(\frac{\lambda}{\tilde{\lambda}}\right)^j
    \]
    is increasing in $j$. \citet[Theorem 1.C.1]{SS:07} implies $\{\omega_j(\tilde{\lambda}), j \in \mathbb{N}_0\}$ is smaller than $\{\omega_j(\lambda), j \in \mathbb{N}_0\}$ in the usual stochastic ordering. Since $\int \psi(x) F_j(dx)$ is a decreasing function of $j$, 
    \[
     \sum_{j=0}^\infty \omega_j(\lambda) \int \psi(x) F_j (dx) \leq  \sum_{j=0}^\infty \omega_j(\tilde{\lambda}) \int \psi(x) F_j (dx).
    \]
    Hence, for any convex function $\psi$, $\mathbb{E} \psi(X_i) \leq \mathbb{E} \psi(\tilde{X}_i)$, provided the expectations exist. 
\end{proof}

The assumption that the distribution of $C$ is log-concave is not a necessary condition for $ \alpha < \tilde{\alpha}$ to imply $M(\infty) \leq_{\rm Lorenz} \tilde{M}(\infty)$. Suppose $C$ is supported on $\{0,1,2\}$. The distribution of $C$ is not log-concave if $\pi_C(1)^2 < \pi_C(0)\pi_C(2)$. The first example following Theorem \ref{thm:rep} showed that for any distribution on $C$ supported on $\{0,1,2\}$, $F_1 \leq_{\rm cx} F_0$. Since $\beta_j = 0$ for all $j \geq 2$, there are no component distributions $F_j$ for $j \geq 2$. Therefore, the same proof used for Theorem \ref{thm:alpha} shows that $ \alpha < \tilde{\alpha}$ implies $M(\infty) \leq_{\rm Lorenz} \tilde{M}(\infty)$ whenever $C$ is supported on $\{0,1,2\}$. Even in cases where the distribution of $C$ is log-convex, the measures of aggregation CV and $1-\text{prevalence}$ can be increasing (Figure \ref{fig:alpha}), suggesting $M(\infty) \leq_{\rm Lorenz} \tilde{M}(\infty)$ may hold in general.

\begin{figure}[h]
    \centering
    \includegraphics[width=\linewidth]{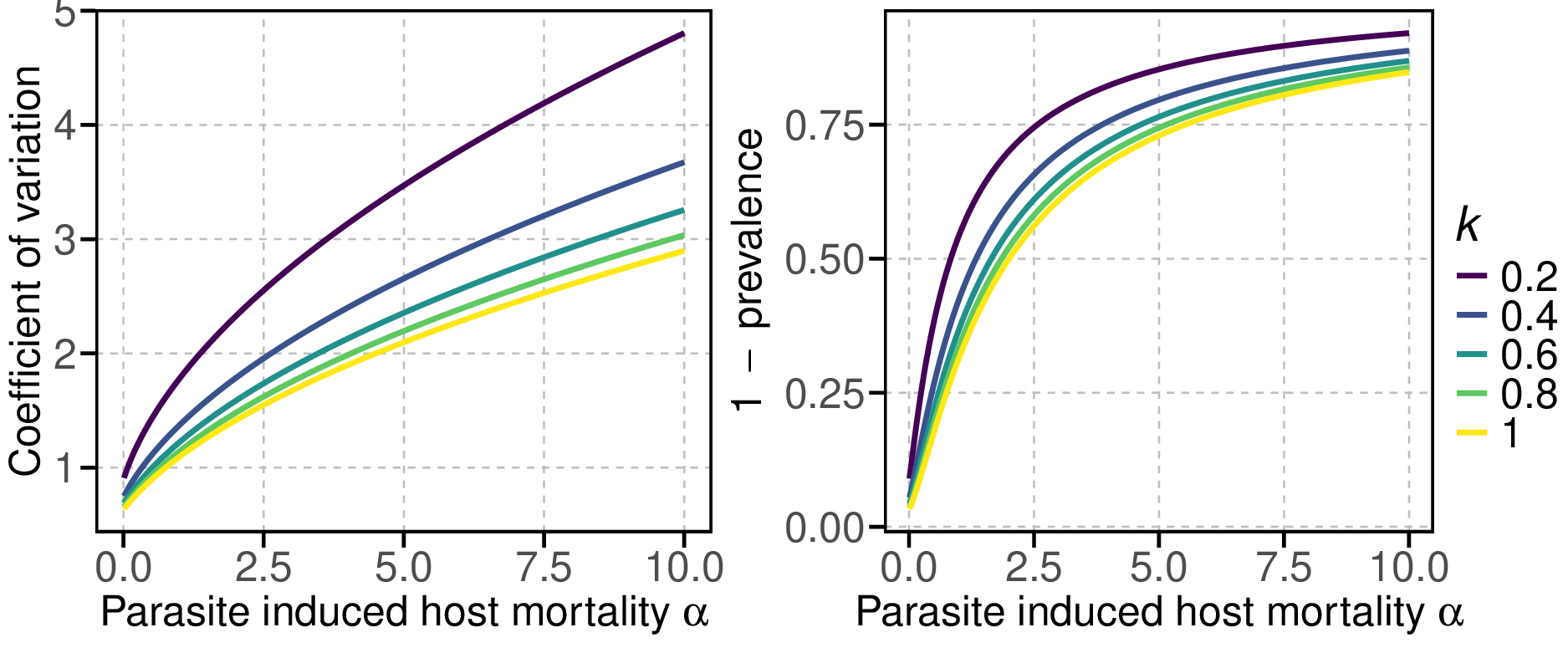}
    \caption{Plot of the coefficient of variation (left) and $1-\text{prevalence}$ (right) calculated for $M(\infty)$ as a function of $\alpha$ with $\phi=5$, $\mu_M=1$ and $C \sim \mathsf{NB}(1,k)$ where $k \in \{0.2, 0.4, \ldots, 1\}$.}
    \label{fig:alpha}
\end{figure}

\subsection{Parasite mortality}
In some respects, the rate of parasite mortality $\mu_M$ has a similar affect on aggregation as the PIHM rate $\alpha$. In the examples of Section \ref{sec:Isham} increasing the parasite mortality rate increases aggregation. Furthermore, in the absence of PIHM, increasing parasite mortality always increases aggregation.

\begin{corollary} \label{cor:mu}
    If $ \mu_M < \tilde{\mu}_M$ and all other model parameters are equal with $\alpha =0$, then $M(\infty) \leq_{\rm Lorenz} \tilde{M}(\infty)$.
\end{corollary}

\begin{proof}
    If $\alpha = 0$, then $\lambda = \tilde{\lambda} =1$ for any values of $\mu_M$ and $\tilde{\mu}_M$. Let 
    \[
    N \sim \mathsf{Poisson}\left(\frac{\phi G^\prime_C(1)}{\mu_M}\right) \quad \text{and} \quad \tilde{N} \sim \mathsf{Poisson} \left(\frac{\phi G^\prime_C(1)}{\tilde{\mu}_M}\right).
    \] 
    By Theorem \ref{thm:rep}, $M(\infty) \stackrel{d}{=} \sum_{i=1}^N X_i$ and $\tilde{M}(\infty) \stackrel{d}{=} \sum_{i=1}^{\tilde{N}} X_i$, where $X_1, X_2, \ldots$ is a sequence of independent and identically distributed random variables. As $\mu_M < \tilde{\mu}_M$, the result follows from Lemma \ref{lem:CP}.
\end{proof}

In the presence of PIHM, increasing $\mu_M$ increases $\lambda$. This implies that more weight is placed on the component distributions $F_j$ with large $j$ in the compound Poisson representation of $M(\infty)$. If the distribution of $C$ is log-convex, then $F_j \leq_{\rm cx} F_{j+1}$ by Theorem \ref{thm:Fcx}. Essentially the same arguments used in the proof of Theorem \ref{thm:alpha} can be applied to show that increasing $\mu_M$ increases aggregation when $C$ has a log-convex distribution.

\begin{theorem} \label{thm:mu}
     Assume the distribution of $C$ is log-convex. If $\mu_M < \tilde{\mu}_M$ and all other model parameters are equal, then $M(\infty) \leq_{\rm Lorenz} \tilde{M}(\infty)$.
\end{theorem}

Unfortunately, we cannot expect this behaviour to hold in general. Figure \ref{fig:mu} shows the coefficient of variation and $1-\text{prevalence}$ as a function of the parasite mortality rate $\mu_M$ with $\alpha = 1, \phi=5$, and $C \sim \mathsf{Poisson}(m_C)$ with $m_C \in \{4,5,6\}$ so the distribution of $C$ is log-concave. We see aggregation initially decreases as the parasite mortality rate increases in these cases. Furthermore, the minimums of the coefficient of variation and $1-\text{prevalence}$ occur for different values of $\mu_M$. Recall that if $X \leq_{\rm Lorenz} Y$, then $I(X) \leq I(Y)$ for all measures $I(\cdot)$ respecting the Lorenz order. Therefore, for at least some values of $\mu_M$ the distribution of parasite burden cannot be compared in the Lorenz order. 

\begin{figure}[h]
    \centering
    \includegraphics[width=\linewidth]{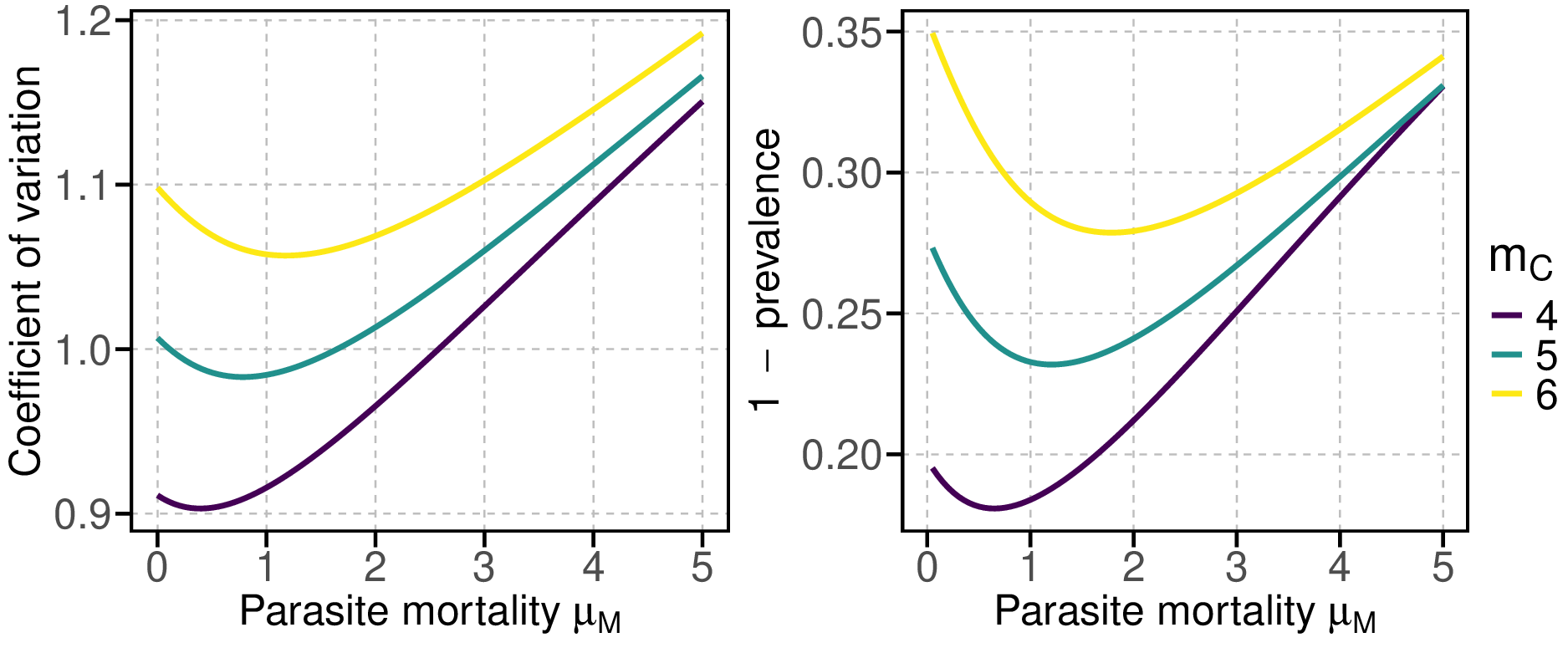}
    \caption{Plot of the coefficient of variation and $1-\text{prevalence}$ calculated for $M(\infty)$ as a function of $\mu_M$ with $\phi=5$, $\alpha=1$, and $C \sim \mathsf{Poisson}(m_C)$ where $m_C \in \{4,5,6\}$.}
    \label{fig:mu}
\end{figure}

\subsection{Rate of infectious contacts}
The rate of infectious contacts has no effect on VMR. However, we expect an increase in $\phi$ to decrease aggregation as $\text{CV} \propto 1/\sqrt{\phi}$. In fact, this follows immediately by applying Lemma \ref{lem:CP} to the compound Poisson representation in Theorem \ref{thm:rep}.

\begin{corollary}
    If $\phi < \tilde{\phi}$ and all other model parameters are equal, then $\tilde{M}(\infty) \leq_{\rm Lorenz} M(\infty)$.
\end{corollary}
    
As a tractable means of incorporating between-host heterogeneity, Isham allowed the parameter $\phi$ for each host to be an independent realisation of a random variable $\Phi$. This increases VMR for the population over that for a single host. This form of between-host heterogeneity also increases aggregation in the sense of the Lorenz order.

\begin{theorem}
    If $\Phi \leq_{\rm cx} \tilde{\Phi}$, then $M(\infty) \leq_{\rm cx} \tilde{M}(\infty)$. 
\end{theorem}

\begin{proof}
    Assume $\lambda <1$. Let $N(\cdot)$ be a Poisson process with rate $(1-G_C(\lambda))/\alpha$. By Theorem \ref{thm:rep} we need to show
    \[
    \sum_{i=1}^{N(\Phi)} X_i \leq_{\rm cx} \sum_{i=1}^{N(\tilde{\Phi})} X_i,
    \]
    where $X_1,X_2,\ldots$ is a sequence of independent and identically distributed random variables. As the Poisson distribution is a regular exponential family distribution with expectation linear in the mean, \citet[Proposition 2]{Schweder:1982} implies $\mathbb{E} \psi(N(\phi))$ is a convex function of $\phi$. Applying \citet[Theorem 3.A.21]{SS:07}, we see $N(\Phi_1) \leq_{\rm cx} N(\Phi_2)$. The result now follows from \citet[Theorem 3.A.14]{SS:07}. To prove the case where $\lambda = 1$, only the rate of the Poisson process $N(\cdot)$ changes.
\end{proof}

\subsection{Distribution of $C$}
In the absence of PIHM, Isham's model reduces to a special case of the Tallis-Leyton model \citep{TL:1969}. Our previous analysis \citep{McVinish:2025b} showed that increased aggregation in the distribution of $C$ increases aggregation in the parasite load of hosts. In particular, if $C$ and $\tilde{C}$ are related through binomial thinning, that is  $G_{\tilde{C}}(z) = G_C(1-p + pz)$ for some $p \in (0,1)$, then $M(\infty) \leq_{\rm cx} \tilde{M}(\infty)$ \citep[Theorem 4]{McVinish:2025b}. Note that in this case $C \leq_{\rm Lorenz} \tilde{C}$. 

As we have seen for parasite mortality, the presence of PIHM makes the role of the distribution of $C$ more difficult to analyse. We are only able to make some comments in the case where $C\sim\mathsf{NB}(m_C,1)$, for which $M(\infty)$ has the negative binomial distribution (\ref{eq:negative-binomial}). An examination of the standard aggregation measures shows that the effect of $m_C$ is not monotone. Direct calculation shows the square coefficient of variation is  
\[
\text{CV}^2 = \frac{\mu_M + 2\alpha}{\phi} + \frac{\alpha}{\phi} m_C + \frac{\mu_M + \alpha}{\phi m_C}.
\]
This function is decreasing in $m_C$ on $(0, (\mu_M+\alpha)/\alpha]$ and increasing on $[(\mu_M+\alpha)/\alpha,\infty)$. The probability that a host is parasite free is
\[
\mathbb{P}(M(\infty) = 0) = (1+m_C)^{-\phi/(\mu_M + \alpha + \alpha m_C)}
\]
and this can be shown to have similar behaviour to CV. We also evaluate the Gini and Pietra indices numerically using the expressions in \citet{ML:2024}. Figure \ref{fig:C} shows that each of these aggregation measures are initially decreasing in $m_C$ but then become increasing for sufficiently large $m_C$. 
However, the minimum occurs at different values of $m_C$ for different aggregation measures which implies that for at least some values of $m_C$ the distribution of parasite load cannot be compared in the Lorenz order.

\begin{figure}
    \centering
    \includegraphics[width=0.6\linewidth]{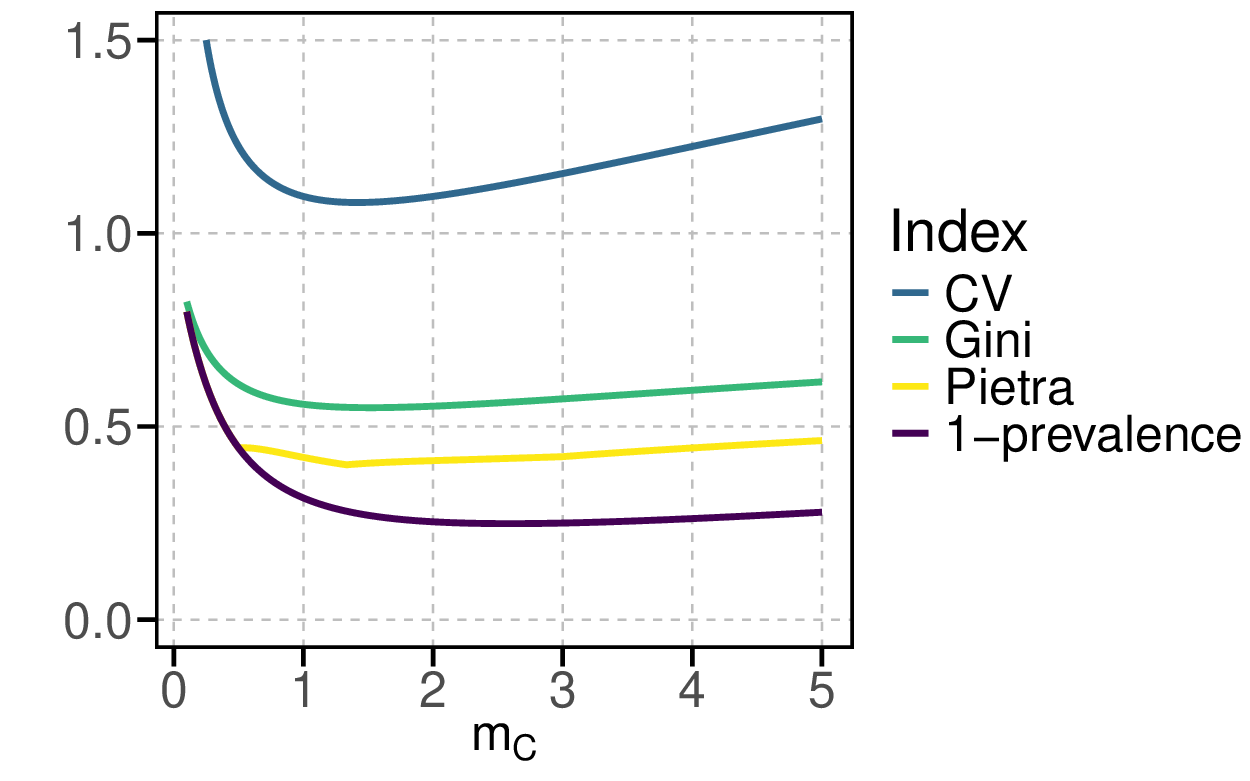}
    \caption{Plot of aggregation measures calculated for $M(\infty)$ with $\phi=5$, $\alpha=1$, $\mu_M=1$, and $C \sim \mathsf{NB}(m_C,1)$.}
    \label{fig:C}
\end{figure}

\section{Conclusion}

This study has examined Isham’s model through the lens of the Lorenz order as the definition of aggregation. By adopting this perspective, we provide new theoretical insight into how parasite-induced host mortality (PIHM) and parasite mortality shape aggregation, complementing the growing literature that applies Lorenz-based measures, particularly the Gini index, to parasitology data. Despite these advances, several questions remain open.

First, our results show that PIHM increases aggregation when the distribution of clump size $C$ is log-concave or is supported on $\{0,1,2\}$. Numerical evidence (Figure \ref{fig:alpha}) suggest this effect may hold under weaker conditions on $C$, but further work is needed to determine whether PIHM consistently increases aggregation across all plausible distributions of $C$.

Second, for log-convex $C$, parasite mortality tends to increase aggregation. However, for other distributions the relationship is non-monotonic, with common aggregation measures showing an initial decrease before a subsequent increase. This pattern could have biological significance, particularly for parasites in paratenic hosts, where parasite mortality rates are typically very low.  A closer examination is required to establish whether this non-monotonic behaviour is confined to certain measures of aggregation or reflects a genuine decrease in the Lorenz order.

Finally, changes in the functional form of PIHM or parasite mortality may substantially alter their affect on aggregation. This was observed by \citet{BP:00} for VMR with $C=1$  almost surely, and there is no reason to expect Lorenz-based measures to  behave differently. Understanding these dependencies remains an important direction for future research.

\bigskip
\noindent {\bf Funding information:} There are no funding bodies to thank relating to this creation of this article.

\bigskip
\noindent {\bf Competing interests:} There were no competing interests to declare which arose during the creation of this article.

\bigskip
\noindent {\bf Use of artificial intelligence (AI) tools:} Microsoft Copilot (GPT-5) was used in the revision of some text and finding the example used in Figure 2.




%
%
%
%

\bibliographystyle{plainnat}
\bibliography{parasites}

\end{document}